\documentclass[12pt]{amsart}
\usepackage{hyperref}
\usepackage{times}
\usepackage[utf8]{inputenc}
\usepackage[capitalize,noabbrev]{cleveref}
\usepackage[backgroundcolor=yellow!10]{todonotes}

\usepackage{tikz,ytableau}
\usepackage[config, labelfont={normalsize}]{caption,subfig}

\usepackage{graphicx}

\usepackage{gensymb}

\newtheorem{proposition}{Proposition}
\numberwithin{proposition}{section}
\newtheorem{lemma}[proposition]{Lemma}
\newtheorem{corollary}[proposition]{Corollary}
\newtheorem{definition}[proposition]{Definition}
\newtheorem*{definition*}{Definition}
\newtheorem{theorem}[proposition]{Theorem}
\newtheorem{claim}[proposition]{Claim}
\newtheorem{problem}[proposition]{Problem}
\newtheorem*{theorem*}{Theorem}

\theoremstyle{example}

\newtheorem*{example*}{Example}

\theoremstyle{remark}
\newtheorem{remark}[proposition]{Remark}

\newcommand{\C}{\mathbb{C}}

\newcommand{\GLm}{\operatorname{GL}_m}
\newcommand{\GLn}{\operatorname{GL}_n}
\newcommand{\GLk}{\operatorname{GL}_k}

\newcommand{\Sym}[1]{\mathfrak{S}_{#1}}

\newcommand{\mn}{m^n}
\newcommand{\nm}{n^m}

\DeclareMathOperator{\tr}{tr}
\DeclareMathOperator{\Tr}{Tr}
\DeclareMathOperator{\End}{End}

\author{Greta Panova}
\address{UPenn Mathematics Department, 
        209 South 33rd St,
        Philadelphia, PA 19104, USA}
\email{panova@math.upenn.edu}

\author{Piotr Śniady}
\address{
Institute of Mathematics, Polish Academy of Sciences,
\mbox{ul.~Śniadec\-kich 8,} \linebreak 00-956 Warszawa, Poland
} 
\email{psniady@impan.pl}

\title%
[Skew Howe duality and random rectangular Young tableaux]%
{Skew Howe duality \\and random rectangular Young tableaux}



\begin{document}
\maketitle
\begin{abstract}
We consider the decomposition into irreducible components 
of the external power $\Lambda^p(\C^m\otimes \C^n)$ regarded as a
$\GLm\times\GLn$-module. 
Skew Howe duality implies that
the Young diagrams from each pair $(\lambda,\mu)$ 
which contributes to this decomposition turn out to be conjugate
to each other, i.e.~$\mu=\lambda'$. 
We show that the Young diagram $\lambda$ which corresponds to a randomly selected
irreducible component $(\lambda,\lambda')$ has the same distribution
as the Young diagram which consists of the boxes with entries $\leq p$
of a random Young tableau of rectangular shape with $m$ rows and $n$ columns.
This observation allows treatment of the asymptotic version of this decomposition
in the limit as $m,n,p\to\infty$ tend to infinity.
\end{abstract}

\section{Introduction}
\label{sec:problem}

\subsection{The problem}
In this note we address the following question.
\begin{quotation}
\emph{Consider 
\begin{equation}
\label{eq:decomposition}
 \Lambda^p(\C^m\otimes \C^n)=\bigoplus_{\lambda}S^{\lambda}\C^m\otimes S^{\lambda'}\C^n   
\end{equation}
as a $\GLm\times \GLn$-module. 
[\dots]
I would
like any information on the shapes of pairs of Young diagrams $(\lambda,\lambda')$ that
give the largest contribution to the dimension asymptotically.
[\dots] 
}

\hfill Joseph M.~Landsberg \cite{Landsberg2012}\footnote{The question of Landsberg is reproduced here in a slightly
redacted version. In particular, the original question considered only the special case $m=n$.}
\end{quotation}

\bigskip

Above, $S^{\lambda}\C^m$ denotes the Schur functor applied to $\C^m$ or, in other words, 
the irreducible representation of the general linear group $\GLm$ 
with the highest weight $\lambda$.
The sum in \eqref{eq:decomposition} runs over Young diagrams 
$\lambda\subseteq n^m$ with $p$ boxes, and such that
the number of rows of $\lambda$ is bounded from above by $m$, and the number of columns of $\lambda$
is bounded from above by~$n$.
The decomposition \eqref{eq:decomposition} is nowadays referred to as 
\emph{skew Howe $(\GLm\times\GLn)$-duality},  cf.~\cite[Theorem 4.1.1]{Howe}.
Even though \eqref{eq:decomposition}
provides full information about the decomposition into irreducible components, it is not very 
convenient for answering such asymptotic questions, see the introduction to the work of Biane \cite{Biane1998}
for discussion of difficulties related to similar problems.

Despite improvements in the understanding of asymptotic problems related to the representation theory
of the general linear groups $\GLm$ \cite{Biane1995,Bufetov2015,Collins2016},
we do not see generic tools which would be suitable for investigation of the external 
power \eqref{eq:decomposition}.

\subsection{Motivations: Geometric Complexity Theory}
Besides the natural interest in the question as a problem in asymptotic representation theory, 
this question is also relevant within Geometric Complexity Theory (GCT). 
The decomposition appears in the study of the complexity of matrix multiplication~\cite{Landsberg2015} 
and, in particular, 
in the study of the border rank of the matrix multiplication tensor as a standard measure of complexity. 
A lower bound for the border rank is obtained from the rank of a particular linear map, 
whose kernel can be decomposed as a $\operatorname{GL}(V)\times \operatorname{GL}(W)$ representation. 
The general approach in GCT would be to study the irreducible components for polynomials to play the role of 
``obstruction candidates'' and, 
depending on the precise setup, the multiplicities would show where to find the obstructions.

\subsection{The main result}
A partial answer to the question of Landsberg which we give in the current paper
is based on a simple result which transforms the original problem
into a question about the representation theory of the symmetric group $\Sym{p}$
for which more asymptotic tools are available, see \cref{sec:applications} below.

We state our main result in two equivalent versions which are of quite distinct flavors:
\begin{itemize}
   \item 
as \cref{theo:enum} which is conceptually simpler and is a purely enumerative statement
which relates some dimensions of the representations of the general linear groups $\GLk$
to the dimensions of some representations of the symmetric groups $\Sym{p}$, and
  \item 
as \cref{theo:main} which is a probabilistic statement which relates the distribution of 
a random irreducible component of the external power \eqref{eq:decomposition} 
to the distribution of a random irreducible component
of a certain representation of the symmetric group $\Sym{p}$.
This second formulation is more convenient for addressing Landsberg's problem.
\end{itemize}
The proof of \cref{theo:enum} is shorter, but the proof of \cref{theo:main}
might be advantageous for some readers who prefer more representation-theoretic
viewpoint.

\subsection{The main result: the enumerative version}
Let $m,n\geq 1$ be integers 
and let $\lambda\subseteq n^m$ be a Young diagram with $p$ boxes which has at most $m$ rows
and at most $n$ columns.
We denote by $\nm$ the rectangular Young diagram with $m$ rows and $n$ columns.
We denote by $f^{\lambda}$ the dimension of the irreducible representation of $\Sym{p}$ 
corresponding to the Young diagram $\lambda$. Note that
the skew Young diagram $\nm / \lambda$ is a rotation by $180\degree$
of a certain Young diagram therefore 
it defines an \emph{irreducible} representation of $\Sym{p}$.

\begin{theorem}
\label{theo:enum}
For a Young diagram $\lambda\subseteq n^m$ with $p$ boxes
we have the following relationship between dimensions of representations
of $\GLm$, $\GLn$ and $\Sym{p}$:
\begin{equation}
\label{eq:main-identity}
 \frac{ \dim(S^{\lambda} \C^m) \dim(S^{\lambda'}\C^n) }{ \dim \Lambda^p (\C^m \otimes \C^n )}
=
\frac{ f^\lambda f^{n^m / \lambda} }{ f^{n^m} }.
\end{equation}
\end{theorem}

Our proof of this result (see \cref{sec:enum-proof}) 
will be based on al\-ge\-bra\-ic com\-bi\-na\-to\-rial manipulations with the
hook-length formula and the hook-content formula.

\subsection{Bijective proofs?}

\cref{theo:enum} implies the following result.
\begin{claim}
\label{claim:fraction}
For all integers $n,m\geq 1$ and $0\leq p\leq nm$ the fraction
\[  
C_{n,m,p}:=\frac{ f^\lambda f^{n^m / \lambda} }{  \dim(S^{\lambda} \C^m) \dim(S^{\lambda'}\C^n) }\]
is a constant which does not depend on the choice of a Young diagram $\lambda\subseteq n^m$ with $p$ boxes.
\end{claim}

Conversely, \cref{claim:fraction} implies \cref{theo:enum} since 
\[  
C_{n,m,p} \sum_{\lambda\vdash p} \dim(S^{\lambda} \C^m) \dim(S^{\lambda'}\C^n) := \sum_{\lambda\vdash p} f^\lambda f^{n^m / \lambda} \]
implies
\[  
C_{n,m,p}\  \Lambda^p(\C^m\otimes \C^n) := f^{n^m} \]
(the left-hand side is an application of skew Howe duality \eqref{eq:decomposition})
which determines uniquely the constant $C_{n,m,p}$.

\medskip

This observation opens the following challenging problem.
\begin{problem}
For a pair of Young diagrams $\lambda,\mu\subseteq n^m$, each with $p$ boxes,
find a \emph{bijective} proof of the identity
\begin{multline}
\label{eq:bijective-problem1}
 \dim(S^{\lambda} \C^m)\ \dim(S^{\lambda'}\C^n) \ f^{\mu} f^{n^m / \mu} 
=  \\ \dim(S^{\mu}\C^m)\ \dim (S^{\mu'}\C^n)\
 f^\lambda\ f^{n^m / \lambda} 
\end{multline}
which is clearly equivalent to \cref{claim:fraction} and thus to \cref{theo:enum}.
\end{problem}

In fact, it would be enough to find such a bijection in the special case when $\mu$ is obtained from $\lambda$
by a removal and an addition of a single box.

\subsection{The main result: the probabilistic version}

\begin{theorem}
\label{theo:main}
Let $m,n\geq 1$ and $0\leq p\leq mn$ be integer numbers. 

The random irreducible component of \eqref{eq:decomposition}
corresponds to a pair of Young diagrams $(\lambda,\lambda')$,
where $\lambda$ has the same distribution as
the Young diagram which consists of the boxes with entries $\leq p$ 
of a uniformly random Young tableau with rectangular shape $n^m$ 
with $m$~rows and $n$ columns.

Alternatively: the random Young diagram $\lambda$ has the same distribution
as a Young diagram which corresponds to a random irreducible component of
the restriction $V^{n^m}\big\downarrow^{\Sym{mn}}_{\Sym{p}}$
of the irreducible representation $V^{n^m}$ of the symmetric group
$\Sym{mn}$ which corresponds to the rectangular diagram $n^m$. 
\end{theorem}

Above, when we speak about \emph{random irreducible component of a representation}
we refer to the following concept.

\begin{definition}
\label{def:random-representation}
For a representation $V$ of a group $G$ we consider its decomposition into
irreducible components
\[ V = \bigoplus_{\zeta \in \widehat{G} } m_\zeta V^\zeta, \]
where $m_\zeta\in\{0,1,\dots\}$ denotes the multiplicity of $V^\zeta$ in $V$.
This defines a probability measure $\mathbb{P}_V$ 
on the set $\widehat{G}$ of irreducible representations 
given by
\[ \mathbb{P}_V(\zeta)=\mathbb{P}^G_V(\zeta) 
:= \frac{m_\zeta \operatorname{dim} V^{\zeta}}{\operatorname{dim} V}.\]
\end{definition}

With this definition in mind, each side of the identity
\eqref{eq:main-identity} from \cref{theo:enum} can be interpreted as 
the probability that an 
appropriate random Young diagram (which appears in \cref{theo:main})
has a specified shape. This provides the link between \cref{theo:enum}
and \cref{theo:main}.

\subsection{Application: back to Landsberg's problem}
\label{sec:applications}

\begin{figure}
\centering
\includegraphics[scale=0.8]{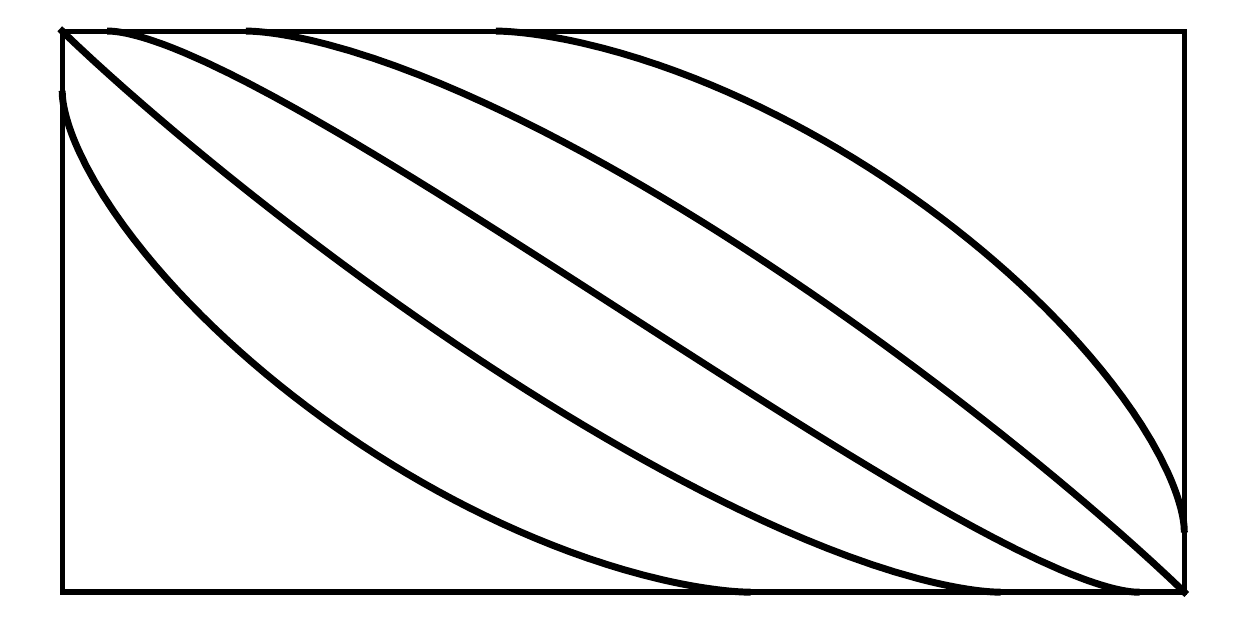}
\caption{Asymptotic limit shapes of typical random Young diagrams
which appear in \cref{theo:main}, cf.~\cite[Figure 3]{Pittel2007}.
We draw Young diagrams in the French convention.
In this example the rectangle ratio is given by $\frac{m}{n}=\frac{1}{2}$.
The limit curves correspond to $\frac{p}{mn}\in\left\{ \frac{1}{6},\frac{2}{6},
\frac{3}{6},\frac{4}{6},\frac{5}{6}\right\}$.}
\label{fig:romik}
\end{figure}

The problem of Landsberg is exactly a 
question about the statistical properties of the random Young diagram
$\lambda$ which appears in \cref{theo:main}.
This result gives an alternative description of $\lambda$
in terms of the representation theory of the symmetric groups $\Sym{p}$
in which many asymptotic problems have well-known answers.
Fortunately, this happens to be the case for the problem of understanding the 
restriction of irreducible representations which we encounter
in \cref{theo:main}.

In particular, the law of large numbers for the corresponding random
Young diagrams has been proved in much wider generality by Biane 
\cite[Theorem 1.5.1]{Biane1998}
using the language of \emph{free cumulants} of Young diagrams.
The asymptotic Gaussianity of their fluctuations around the limit shape has been proved 
by the second-named author
\cite[Example 7 combined with Theorem 8]{Sniady2006}
using the same language.

In the specific case of the restriction $V^{n^m}\big\downarrow^{\Sym{mn}}_{\Sym{p}}$
which is in the focus of the current paper, the above-mentioned generic tools
\cite{Biane1998,Sniady2006} can be applied in the scaling when
$m,n,p\to\infty$ tend to infinity
in such a way that the rectangle ratio $\frac{m}{n}$ converges to a strictly positive limit
and the fraction $\frac{p}{mn}$ converges to some limit.
Pittel and Romik \cite{Pittel2007} have worked out this specific example
and, among other results, found explicit asymptotic limit shapes of typical Young diagrams
which contribute to such representations, see \cref{fig:romik}.
In the light of \cref{theo:main}, the above references provide a partial answer to the question of
Landsberg.

\subsection{Hypothetical extensions of \cref{theo:main}}

The formulation of \cref{theo:main} might suggest that it is a special case of
a more general result. We state it concretely as the following problem.

\begin{problem}
Find a natural \emph{quantum random walk} 
(in the spirit of Biane \cite{Biane1991})
on the set of irreducible representations (of some group? of some algebra?) 
with the property that 
the probability distribution on the set of paths of this random walk
can be identified (via some hypothetical analogue of \cref{theo:main})
with the uniform distribution on the set of standard Young tableaux
of rectangular shape $n^m$.
\end{problem}

\section{Proof of \cref{theo:enum}}
\label{sec:enum-proof}

First, we give an enumerative proof of \cref{theo:enum} using the classical dimension formulas:
the hook-length formula for $f^{\nu}$
$$f^{\nu}= \frac{ |\nu|!}{\prod_{\Box \in \nu} h_\Box} =: \frac{ |\nu|!}{H_\nu},$$
where $H_\nu$ denotes the product of hook lengths in $\nu$,
and the hook-content formula for the dimensions of representations of $\GLk$:
\begin{equation}
\label{eq:dimension-glk}
\dim S^\nu \C^k = s_{\nu}(1^k) = \frac{ \prod_{\Box\in \nu} (k+c(\Box))}{\prod_{\Box \in \nu}h_\Box},   
\end{equation}
where $h_\Box$ is the hook-length of a box $\Box$ in the diagram of $\nu$, and the content $c(\Box) = j-i$, 
if $\Box=(i,j)$ is at row $i$ and column $j$ of the diagram. 
Here $s_\nu(x_1,\ldots,x_k)$ is the corresponding Schur function.

\begin{claim}
\label{claim:complementary}
We have that 
\begin{equation}
\label{eq:complementary}
 s_{\nu}(1^n) = s_{m^n / \nu}(1^n) = s_{\bar{\nu}}(1^n)   
\end{equation}
for any $\nu \subseteq m^n$, 
where $\bar{\nu} = (m-\nu_n,m-\nu_{n-1},\ldots,m-\nu_1)$ is the complementary partition.   
\end{claim}
\begin{proof}
By the combinatorial definition of Schur functions, 
we have that $s_{\nu}(1^n)$ is the number of SSYTs of shape $\nu$ and entries $1,\ldots,n$. 
Consider such a SSYT as sitting inside the rectangular $m^n$ box containing $\nu$. 
Each such a SSYT $T$ corresponds to a complementary SSYT $P$ of
shape $\bar{\nu}$ and entries $1,\ldots,n$ 
via the bijection which we describe in the following.

\begin{figure}[b]
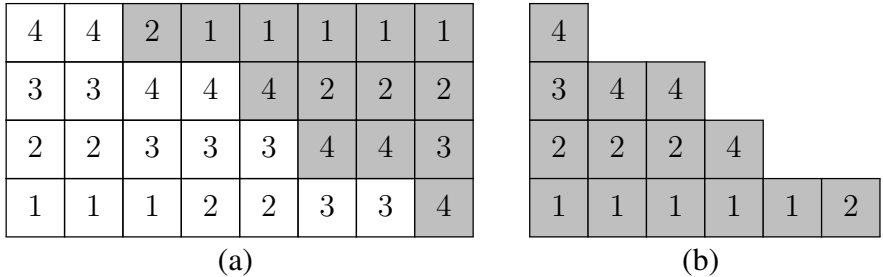

    \subfloat[]{
        \centering
	      \ytableausetup{boxsize=4ex}
	      \ytableaushort{
	      % a4a4b3b1b1b1b1b1,
	      % a3a3a4a4b4b2b2b2,
	      % a2a2a3a3a3b4b3b3,
	      % a1a1a1a2a2a3a3b4
	      44{*(lightgray)2}{*(lightgray)1}{*(lightgray)1}{*(lightgray)1}{*(lightgray)1}{*(lightgray)1},
	      3344{*(lightgray)4}{*(lightgray)2}{*(lightgray)2}{*(lightgray)2},
	      22333{*(lightgray)4}{*(lightgray)4}{*(lightgray)3},
	      1112233{*(lightgray)4}
	      }
           \label{bijection:A}    
    }
\quad
    \subfloat[]{
        \centering
	      \ytableausetup{boxsize=4ex}
	      \ytableaushort{
{*(lightgray)4},	
{*(lightgray)3}{*(lightgray)4}{*(lightgray)4},
{*(lightgray)2}{*(lightgray)2}{*(lightgray)2}{*(lightgray)4},
{*(lightgray)1}{*(lightgray)1}{*(lightgray)1}{*(lightgray)1}{*(lightgray)1}{*(lightgray)2}
}
\label{bijection:B}
  }
\caption{An example of the bijection from the proof of \cref{claim:complementary}
with $m=8$, $n=4$, and $\nu=(7,5,4,2)$.
\protect\subref{bijection:A}
The original tableau is in white, 
and the (not rotated yet) complementary tableau of skew shape $m^n/{\nu}$ is in gray.
\protect\subref{bijection:B}
The complementary tableau of shape $\bar{\nu}=(6,4,3,1)$.
}
\label{fig:complementary}
\end{figure}

In a given column $j$ of $m^n$, let $a_1<\cdots<a_{\nu'_j}$ be the entries of $T$ in this column. 
Let $ \{b_1<b_2<\cdots<b_{n - \nu'_j}\}:=\{1,\ldots,n\} \setminus \{a_1,\ldots,a_{\nu'_j}\}$ 
be the remaining numbers in $\{1,\ldots,n\}$. 
Write them in increasing order top to bottom in the column $j$ in $m^n$ above $T$ 
as in Figure~\ref{bijection:A}; note that we use the French convention for drawing Young diagrams.
Rotating the resulting tableau above $T$ by $180\degree$ we obtain a SSYT of shape $\bar{\nu}$ 
with entries in $\{1,\ldots,n\}$ as in Figure~\ref{bijection:B}
(the row inequalities are easily seen to be satisfied).
\end{proof}
\begin{remark}
\cref{claim:complementary} has also a purely representation-theoretic proof:
the left-hand side of \eqref{eq:complementary} is equal to the dimension of the 
representation $S^\nu\C^n$ of $\GLn$ which corresponds to the Young diagram $\nu$
while the right-hand side is equal to the dimension of the tensor product
of the one-dimensional representation $\GLn\ni g \mapsto (\det g)^m$
with the representation contragradient to~$S^\nu\C^n$. Their dimensions are clearly equal.
\end{remark}

We continue the proof of \cref{theo:enum}.
Using the claim for $\nu:=\lambda'$ we have that
\begin{equation}
\label{eq:dim-lambda-comp}
\dim S^{\lambda'}\C^n = s_{\lambda'}(1^n) = s_{\overline{\lambda'}}(1^n) = 
\frac{ \prod_{\Box \in \overline{\lambda'} }(n +c(\Box))}{H_{\overline{\lambda'}}}=
\frac{ \prod_{\Box \in \overline{\lambda'} }(n +c(\Box))}{H_{\bar{\lambda}}}.  
\end{equation}
The last equality follows from the observation that $\overline{\lambda'} = (\bar{\lambda})'$, 
where the second complement is taken in the $n^m$ rectangle and so $H_{\overline{\lambda'} } = H_{\bar{\lambda}}$.

\medskip

\begin{figure}[tb]
\centering
\hspace{-7ex}
\subfloat[]{
\centering
\begin{tikzpicture}[scale=0.75]
\draw[gray,very thin] (0,0) grid (8,6);
\fill[fill=gray,opacity=0.5] (0,5) -- (2,5) --(2,4)--(4,4)--(4,3)--(5,3)--(7,3)--(7,1)--(8,1)--(8,0)--(0,0)--cycle;
\draw[thick] (0,0) rectangle (8,6);
\draw[ultra thick] (0,6) -- (0,5) -- (2,5) --(2,4)--(4,4)--(4,3)--(5,3)--(7,3)--(7,1)--(8,1)--(8,6) --cycle;

\draw[dashed] (5.5,-1)--(5.5,7);
\draw[dashed] (-1,4.5)--(9,4.5);
\node[fill=white] at (5.5,4.5) {$\Box$};
\node at (0.5,-0.5) {$1$};
\node at (7.5,-0.5) {$n$};
\node at (-0.5,0.5) {$1$};
\node at (-0.5,5.5) {$m$};
\node at (8.5,5.5) {$1$};
\node at (8.5,0.5) {$m$};
\node at (7.5,6.5) {$1$};
\node at (0.5,6.5) {$n$};

\node[fill=white] at (-0.5,4.5) {$i$};
\node[fill=white] at (5.5,-0.5) {$j$};
\node[fill=white] at (8.5,4.5) {$j'$};
\node[fill=white] at (5.5,6.5) {$i'$};

\end{tikzpicture}
\label{subfig:rectangleA}
}~\subfloat[]{
\centering
\begin{tikzpicture}[scale=0.75,x={(0cm,-1cm)},y={(-1cm,0cm)}]
\draw[ultra thin,gray!30] (8,6) grid (0,0);
\begin{scope}
\clip (0,6) -- (0,5) -- (2,5) --(2,4)--(4,4)--(4,3)--(5,3)--(7,3)--(7,1)--(8,1)--(8,6) --cycle;
\draw[gray,thin] (8,6) grid (0,0);
\end{scope}
\draw[ultra thick] (0,6) -- (0,5) -- (2,5) --(2,4)--(4,4)--(4,3)--(5,3)--(7,3)--(7,1)--(8,1)--(8,6) --cycle;

\draw[dashed] (5.5,-1)--(5.5,7);
\draw[dashed] (-1,4.5)--(9,4.5);
\node[fill=white] at (5.5,4.5) {$\Box'$};
\node at (8.5,5.5) {$1$};
\node at (8.5,0.5) {$m$};
\node at (7.5,6.5) {$1$};
\node at (0.5,6.5) {$n$};

\node[fill=white] at (8.5,4.5) {$j'$};
\node[fill=white] at (5.5,6.5) {$i'$};

\end{tikzpicture}
\label{subfig:rectangleB}
}

\caption{The relationship between contents within $\lambda$, $\overline{\lambda'}$ and the rectangle $n^m$.
\protect\subref{subfig:rectangleA}
In this example $m=6$, $n=8$, 
$\lambda = (8,7,7,4,2)$ is drawn with gray background.
\protect\subref{subfig:rectangleB}
The partition $\overline{\lambda'} = (5,3,3,3,2,2,1,1)$ is drawn with white background. 
Figure \protect\subref{subfig:rectangleA} shows how $\lambda$ and $\overline{\lambda'}$
are combined together within the rectangle $n^m$ and how their respective coordinate
systems are related.
The content of $\Box'=(i',j')=(3,2)\in\overline{\lambda'}$ is $c(\Box')=2-3=-1$, the same box in $n^m$ is 
$\Box=(i,j)=(5,6)$ with $c(\Box)=1$ and 
$n+c(\Box') = 7 = m+c(\Box)$. }
\label{fig:contents} 
\end{figure}

Consider the partitions $\lambda$ and $\overline{\lambda'}$ as sitting inside $n^m$.
More specifically, any box 
\[ \Box'=(i',j') \in \overline{\lambda'}\] 
corresponds to the box 
\[ \Box=(i,j)=(m+1-j',n+1-i')\in n^m,\] 
see Figure~\ref{fig:contents}, 
and so the content $c(\Box')=c_{\bar{\lambda}'}(\Box')$ of $\Box$ regarded as a box of $\overline{\lambda'}$
fulfills
\[n+c_{\bar{\lambda}'}(\Box') = n+j' -i' = n+(m+1-i)-(n+1-j) = m+c_{n^m}(\Box),\] 
where the content $c(\Box)=c_{n^m}(\Box)$ is taken with respect to the $n^m$ partition. 

\medskip

Since $s_{n^m}(1^m) = 1$ 
(there is only one SSYT of shape $n^m$ and entries $1,\ldots,m$ 
since each column is forced to be $1,\ldots,m$),
it follows from the hook-content formula \eqref{eq:dimension-glk}
that 
\begin{equation}
\label{eq:onedimensional}
 \prod_{\Box \in n^m} (m+c(\Box)) = H_{n^m}.   
\end{equation}

\medskip

Thus by \eqref{eq:dim-lambda-comp}, then by
combining the diagrams $\lambda$ and $\overline{\lambda'}$
into the $n^m$ rectangle,
and by \eqref{eq:onedimensional}
\begin{align*}
\dim S^{\lambda} \C^m \dim S^{\lambda'}\C^n   &=
s_\lambda(1^m)\ s_{\lambda'}(1^n)  \\&=
\frac{ \prod_{\Box \in \lambda}  (m+c(\Box))}{H_\lambda} 
\frac{ \prod_{\Box' \in n^m /\lambda} (m+c(\Box'))}{H_{\bar{\lambda}}} \\
&=\frac{ \prod_{\Box \in n^m} (m +c(\Box) )}{H_\lambda H_{\bar{\lambda} }} = 
\frac{ H_{n^m} }{H_\lambda H_{\bar{\lambda}} }\\
&= \frac{ (mn)! / f^{n^m} }{ \big(p!/ f^\lambda\big)\big( (mn-p)!/f^{n^m/\lambda}\big)} = \binom{mn}{p} \frac{f^{\lambda} f^{n^m/\lambda} }{f^{n^m}}.
\end{align*}

This concludes the proof of \cref{theo:enum}.

\begin{remark}
Relationships between $H_\lambda$, $H_{\bar{\lambda}}$ and Schur function evaluations have also been derived 
by Stanley \cite{Stanley2001} who used them further 
in computations of the normalized symmetric group character corresponding 
to rectangular partitions. 
\end{remark}

\section{Proof of \cref{theo:main}}
\label{sec:repr-proof}

\subsection{Sketch of the proof}
We start by presenting a one-paragraph summary of the proof.
Schur--Weyl duality suggests exploring the link between the structure of the external power
\eqref{eq:decomposition} viewed as a representation of some general linear group $\GLm$
and the structure of the same space \eqref{eq:decomposition}, this time viewed as a
representation of the symmetric group $\Sym{p}$.
Regretfully, the external power \eqref{eq:decomposition} is \emph{not} a representation
of $\Sym{p}$. This approach can be rescued if, instead, we view the external power as a module over
the \emph{center} $Z\C[\Sym{p}]$ of the symmetric group algebra.
The character theory of the symmetric group $\Sym{p}$ can be easily adapted 
to the setting of $Z\C[\Sym{p}]$.
We will show that the characters of $\Lambda^p(\C^m\otimes\C^n)$ 
(for fixed values of $m$ and $n$ 
and for $p$ varying over $\{0,1,\dots,mn\}$) are closely related to each other;
in this way it is enough to identify such a character for $p=mn$.
Yet another application of Schur--Weyl duality shows that this particular character 
is irreducible
and corresponds to the rectangular Young diagram $n^m$.

\medskip

In the remaining part of this section we will present the details of the above
sketch.
For clarity the proof is split into a number of propositions.

\subsection{Schur--Weyl duality. $\GLm$ versus $\Sym{p}$}

\begin{proposition}
\label{prop:SchurWeyl}
The probability distributions of the following two pairs of random Young diagrams are equal:
\begin{itemize}
   \item the pair of random Young diagrams which correspond to a random irreducible component of
$\Lambda^p(\C^m\otimes \C^n)$ \emph{regarded as $\GLm\times \GLn$-module}, and 

   \item the pair of random Young diagrams which correspond to a random irreducible component of
$\Lambda^p(\C^m\otimes \C^n)$ \emph{regarded as $Z\C(\Sym{p})\times Z\C(\Sym{p})$-module}.
\end{itemize}   
\end{proposition}

In the following we shall present the missing details of notation and the proof of this proposition.

\medskip

The tensor power
\begin{equation}
\label{eq:tensor-product}
(\C^m \otimes \C^n )^{\otimes p}  \cong 
   (\C^m)^{\otimes p} \otimes (\C^n)^{\otimes p}   
\end{equation}
carries a natural structure of a $(\GLm\times \Sym{p}) \times (\GLn \times \Sym{p})$-module
and, more generally, a structure of a $(\GLm\times \C[\Sym{p}]) \times (\GLn \times \C[\Sym{p}])$-module:
the general linear group $\GLm$ acts on all corresponding factors $\C^m$
while
the symmetric groups $\Sym{p}$ act by permuting the factors in the tensor product. 
Regretfully, its subspace
\begin{equation}
\label{eq:space}
\Lambda^p(\C^{\otimes m} \otimes \C^{\otimes n})   
\end{equation}
which is in the focus in the current paper
is \emph{not} invariant under the action of the symmetric groups $\Sym{p}$
which are factors in
\[ (\GLm\times \Sym{p}) \times (\GLn \times \Sym{p}) \]
and under the action of the corresponding symmetric group algebras $\C[\Sym{p}]$.
On the bright side, it is not difficult to check that
the space \eqref{eq:space} \emph{is} invariant under the action of the \emph{center}
$Z\C[\Sym{p}]$ of the symmetric group algebra.
Thus both \eqref{eq:tensor-product} and \eqref{eq:space} carry a structure of
a $(\GLm\times Z\C[\Sym{p}]) \times (\GLn \times Z\C[\Sym{p}])$-module.
Our proof of \cref{theo:main} will be based on exploration of this module structure.

\bigskip

\begin{proof}[Proof of \cref{prop:SchurWeyl}]
Consider the decomposition of the module \eqref{eq:space} into irreducible components:
\begin{equation}
\label{eq:decomposition2}
 \Lambda^p(\C^{\otimes m} \otimes \C^{\otimes n}) = 
\bigoplus_{\mu,\lambda,\nu,\pi} 
m_{\mu,\lambda,\nu,\pi}\ 
S^\mu(\C^m) \otimes V^{\lambda} \otimes S^\nu(\C^n) \otimes V^{\pi},     
\end{equation}
where $m_{\mu,\lambda,\nu,\pi}\in\{0,1,\dots\}$ denotes the multiplicity,
$V^{\lambda}$ denotes the irreducible representation of the symmetric group
$\Sym{p}$ which corresponds to the Young diagram $\lambda$
and the sum runs over Young diagrams $\mu,\lambda,\nu,\pi$.

Schur--Weyl duality implies that a decomposition (analogous to \eqref{eq:decomposition2})
of the tensor power \eqref{eq:tensor-product} into irreducible components 
(no matter whether we regard \eqref{eq:tensor-product} as a 
$(\GLm\times \Sym{p}) \times (\GLn \times \Sym{p})$-module
or as a $(\GLm\times Z\C[\Sym{p}]) \times (\GLn \times Z\C[\Sym{p}])$-module)
involves only summands for which $\mu=\lambda$ and $\nu=\pi$.
It follows that the same is true for its 
$(\GLm\times Z\C[\Sym{p}]) \times (\GLn \times Z\C[\Sym{p}])$-submodule 
\eqref{eq:space}, thus
\begin{equation}
\label{eq:decomposition3}
\Lambda^p(\C^{\otimes m} \otimes \C^{\otimes n}) = 
\bigoplus_{\lambda,\pi} 
m_{\lambda,\pi}\ 
S^\lambda(\C^m) \otimes W^{\lambda} \otimes S^\pi(\C^n) \otimes W^{\pi}     
\end{equation}
for some multiplicities $m_{\lambda,\pi}:=m_{\lambda,\lambda,\pi,\pi}$.

The linear span of all irreducible components of \eqref{eq:decomposition3}
which correspond to a given pair of Young diagrams $(\mu,\nu)$ 
remains the same, no matter if we regard \eqref{eq:decomposition3} 
as a $\GLm\times\GLn$-module or
as a $Z\C[\Sym{p}]\times Z\C[\Sym{p}]$ module.
It follows that the corresponding probability distributions are equal.
This concludes the proof of \cref{prop:SchurWeyl}.
\end{proof}

\cref{prop:SchurWeyl} shows that in order to prove \cref{theo:main} 
it is enough to understand the structure
of $\Lambda^{p}(\C^m\otimes \C^n)$ as a $Z\C[\Sym{p}]\times Z\C[\Sym{p}]$-module.
We shall do it in the following.

\subsection{Normalized characters}

In this paper whenever we refer to a \emph{trace} of a
matrix $A=(A_{ij})_{1\leq i,j\leq d}\in M_d(\C)$ we mean
\emph{the normalized trace}
\[ \tr_d A := \frac{1}{d} \sum_{1\leq i\leq d} A_{ii}\]
as opposed to \emph{the non-normalized trace}
\[ \Tr A := \sum_{1\leq i\leq d} A_{ii}.\]
For an operator $A\in\End V$ we denote by $\tr_V A$ its normalized trace,
defined analogously.

Also, by the \emph{character} of a group representation 
$\rho\colon G\to \End V$ we mean \emph{the normalized character} 
$\chi_V\colon G\to \C$ given by
\begin{equation}
\label{eq:group-character}
 \chi_V(g):= \tr_V \rho_g,    
\end{equation}
which is defined in terms of the normalized
trace $\tr_V$.

\subsection{Modules over the center $Z\C[G]$ of the group algebra}
\label{sec:projection}

In the following we will consider the following setup.
We assume that $G$ is a finite group and $W$ is a $G$-module.
We also assume that $\Pi\colon W\to W$ is an idempotent $\Pi^2 = \Pi$
with the property that $\Pi$ commutes with the action of the center $Z\C[G]$ of the
group algebra. We denote by $V:=\Pi W$ the image of $\Pi$.
The space $V$ is invariant under the action of $Z\C[G]$; in other words $V$ can be regarded
as a $Z\C[G]$-module.

We define the \emph{character} of the $Z\C[G]$-module $V$ as a function $\chi_V\colon G\to \C$ given by
\begin{equation}
\label{eq:character1}
 \chi_V(g) := \frac{1}{\operatorname{dim} V} \Tr \big[ \Pi \ \rho_g\big],   
\end{equation}
where $\rho_g\colon W\to W$ denotes the action of $g\in G$ on $W$.

\medskip

Our goal in this proof is to understand $V$ as a $Z\C[G]$-module
and to identify the corresponding character.

\subsection{The key example}
\label{sec:example}

The key example we should keep in mind is the tensor product
\begin{equation}
\label{eq:tensor-productA}
W=W_p :=   (\C^m)^{\otimes p} \otimes (\C^n)^{\otimes p}   
\end{equation}
which carries a natural structure of $G$-module, 
where 
\[G=G_p:=\Sym{p}\times \Sym{p}\] 
is the Cartesian product of the symmetric groups which acts on $W$ by permuting the factors in 
the tensor product.

By rearranging the order of the factors we see that
\begin{equation}
\label{eq:rearranged}
 W\cong (\C^m \otimes \C^n )^{\otimes p}     
\end{equation}
is a tensor power which carries another structure, this time of a $\Sym{p}$-module. 
This action is related to the action from \eqref{eq:tensor-productA} via
the diagonal inclusion of groups
given by 
\begin{equation}
\label{eq:inclusion}
\Sym{p}\ni g\mapsto (g,g) \in \Sym{p} \times \Sym{p}.   
\end{equation}

\medskip

We consider the projection $\Pi$ which is given by the action on \eqref{eq:rearranged} of the central projection 
\begin{align}
\label{eq:antisymm} 
\Pi= \Pi_p &= \frac{1}{p!}  \sum_{g\in\Sym{p}}  (-1)^g g \in Z\C[\Sym{p}] \\ 
\intertext{which under the inclusion \eqref{eq:inclusion} becomes the action
on \eqref{eq:tensor-productA} of the element}
\Pi &= \frac{1}{p!}  \sum_{g\in\Sym{p}}  (-1)^g (g,g) \in \C[\Sym{p}\times \Sym{p}]
\end{align}
which is \emph{not} central.

The image $V=\Pi W$ of this projection is the external power
\[ V= V_p = \Lambda^p(\C^m\otimes \C^n)\]
which is in the focus of the current paper.

\subsection{Modules over the center $Z\C[G]$, revisited}
Consider now a more general situation than in \cref{sec:projection}
in which $V$ is an arbitrary $Z\C[G]$-module, without any additional structure.

\medskip

We can define the \emph{character} $\chi_V\colon G\to\C$ of $V$ by the formula
\begin{equation}
\label{eq:character2}
 \chi_V(g):= \tr_V \rho\left[ \frac{1}{|G|} \sum_{h\in G} h g h^{-1} \right].    
\end{equation}

Note that in the specific setup considered in \cref{sec:projection}
the formulas \eqref{eq:character1} and \eqref{eq:character2} 
define the same function. 

Also, in the specific setup in which 
the structure of a $Z\C[G]$-module on $V$ arises from the structure
of a $G$-module,
the usual character of the group $G$ given by
\eqref{eq:group-character} coincides with the character from \eqref{eq:character2}.

\medskip

The algebra $Z\C[G]$ is commutative, hence each irreducible $Z\C[G]$-module 
is one-dimensional. 
Each irreducible $G$-module $V$, viewed as a 
$Z\C[G]$-module, is a direct sum of a number of copies of a single irreducible
$Z\C[G]$-module $V'$. Their characters are equal: $\chi_{V}=\chi_{V'}$.

More generally, there is a bijective correspondence between
(equivalence classes of) irreducible $Z\C[G]$-modules and
(equivalence classes of) irreducible $G$-modules;
the characters of the corresponding modules are equal.

\medskip

Any $Z\C[G]$-module $V$ defines
(analogously as in \cref{def:random-representation})
a probability measure $\mathbb{P}^{Z\C[G]}_V$ on the set of irreducible representations of $G$.
Note that in the specific situation when the structure of a $Z\C[G]$-module
on $V$ arises from the structure of a $G$-module, the corresponding measures
are equal: $\mathbb{P}^G_V=\mathbb{P}^{Z\C[G]}_V$, no matter if we regard $V$ as
a $G$-module or as a $Z\C[G]$-module.

These probability measures are directly related to the character of the corresponding module:
\[ \mathbb{P}_V(\zeta) = \chi_V(p_\zeta) = \sum_{g\in G} \chi_V(g)\ p_\zeta(g),\]
where $p_{\zeta}\in Z\C[G]$ is the minimal central projection which corresponds
to the irreducible representation $\zeta\in \widehat{G}$.

\subsection{Characters of the external power do not depend on the exponent}
We come back to the specific setup from \cref{sec:example}.
Assume that $0\leq p < p' \leq mn$.
There is a natural inclusion 
$G_p=\Sym{p}\times \Sym{p} \subseteq \Sym{p'}\times \Sym{p'}=G_{p'}$ 
of the corresponding groups which allows us to compare the characters
of the external powers $V_p=\Lambda^p(\C^m\otimes \C^n)$ for various values of the exponent~$p$.
\begin{lemma}
\label{lem:do-not-depend}
For $0\leq p < p' \leq mn$ the character $\chi_{V_p} \colon G_p \to \C$ is equal to the restriction
of the character $\chi_{V_{p'}} \colon G_{p'} \to \C$.
\end{lemma}
\begin{proof}
It is enough to prove this result in the case when $p'=p+1$.   

\medskip

Let $g=(\pi_1,\pi_2)\in G_{p}=\Sym{p}\times \Sym{p}$.
We denote by $\pi_1',\pi_2'\in\Sym{p+1}$ the corresponding permutations from the larger symmetric group;
in this way $g$ corresponds to $g'=(\pi_1',\pi_2')\in G_{p+1}$. 

\medskip

We consider the decomposition
\begin{equation}
\label{eq:decomposition-tensor}
 W_{p+1} = (\C^m \otimes \C^n)^{\otimes (p+1)} = W_p \otimes (\C^m\otimes \C^n).   
\end{equation}
With respect to this decomposition, 
the action of $\rho_{g'}$ on $W_{p+1}$ coincides with the action of $\rho_g \otimes 1$:
\begin{equation}
\label{eq:action}
   \rho_{g'} = \rho_g \otimes 1.
\end{equation}

\medskip

The projection $\Pi_{p+1}$ viewed as in \eqref{eq:antisymm} as an element of $\C[\Sym{p+1}]$
can be written as the product
\begin{equation}
\label{eq:JM}
  \Pi_{p+1} = \frac{1}{p+1}  \big[ 1 - X_{p+1} \big]\ \Pi_{p},   
\end{equation}
where 
\[ X_{p+1}= (1,p+1) + \cdots + (p,p+1) \in \C[\Sym{p+1}] \]
is a Jucys--Murphy element; above $(i,j)\in\Sym{p+1}$ denotes the transposition which interchanges $i$ with~$j$.
We view now \eqref{eq:JM} as an operator acting on \eqref{eq:decomposition-tensor};
with this perspective
\begin{equation}
\label{eq:JM2}
 \Pi_{p+1} = \frac{1}{p+1} \big[ 1 - X_{p+1} \big] \big( \Pi_{p} \otimes 1\big)   
\end{equation}
is an element of $\End (W_p) \otimes \End (\C^m\otimes \C^n)$.

\medskip

A direct calculation on the elementary tensors shows that 
application of the (non-normalized) trace $\Tr\colon \End (\C^m\otimes \C^n) \to\C$ 
to the second factor
in \eqref{eq:JM2} yields a multiple of identity:
\[ (1\otimes \Tr) \frac{1}{p+1} \big[ 1 - X_{p+1} \big]= \frac{mn-p}{p+1}.  \]
By combining this idea with \eqref{eq:action}, \eqref{eq:JM2}
it follows that
\[ \chi_{V_{p+1}} (g') = \frac{1}{\binom{mn}{p+1}} \Tr \left[ \Pi_{p+1}\ \rho_{g'} \right]= 
\frac{1}{\binom{mn}{p}} \Tr \left[ \Pi_p\ \rho_{g} \right] = \chi_{V_{p}} (g)\]
which concludes the proof.
\end{proof}

\subsection{The character of $\Lambda^p(\C^m\otimes \C^n)$}

Until now we considered the symmetric group $\Sym{p}$ as
the diagonal subgroup of $\Sym{p}\times \Sym{p}$ via \eqref{eq:inclusion}.
In the following we take a different perspective and we shall view the symmetric group
\[ \Sym{p} \cong \Sym{p} \times \{\operatorname{id}\} \subseteq \Sym{p}\times\Sym{p} = G_p \]
as the first factor in the Cartesian product.

In this way the space $W_p=(\C^m)^{\otimes p} \otimes (\C^n)^{\otimes p} $ 
has a structure of a $\Sym{p}$-module
and $V_p=\Lambda^p(\C^m\otimes \C^n)$ a structure of a $Z\C[\Sym{p}]$-module.

\begin{corollary}
\label{coro:character-restriction}
The character $\chi_{V_p} \colon \Sym{p} \to \C$ of $Z\C[\Sym{p}]$-module $V_p=\Lambda^p(\C^m\otimes \C^n)$
is equal to the restriction
of the irreducible character $\chi_{\mn}\colon \Sym{mn}\to \C$
of the symmetric group $\Sym{mn}$
which corresponds to the rectangular Young diagram $\mn$.   
\end{corollary}
\begin{proof}
In the light of \cref{lem:do-not-depend} it is enough to prove this result
for the maximal possible value $p=mn$ and to show the equality
\begin{equation}
\label{eq:postponed}
 \chi_{V_{mn}} = \chi_{\mn}.  
\end{equation}
We shall do it in the following.

For $p=mn$ the external power
$V_{mn}$ is a one-dimensional representation of
$\GLm\times \GLn$ which corresponds to a pair of Young diagrams 
\begin{equation}
\label{eq:mn}
(\mn, \nm).   
\end{equation}
The corresponding random pair of Young diagrams (associated via \cref{def:random-representation}) 
is deterministic, equal to \eqref{eq:mn}.
\cref{prop:SchurWeyl} shows that if we view $V_{mn}$
as a $Z\C[\Sym{mn}]\times Z\C[\Sym{mn}]$-module, the corresponding pair of random Young diagrams
is also deterministic, equal to \eqref{eq:mn}.
If we regard $V_{mn}$ as $Z\C[\Sym{mn}]$-module, this implies that its character is equal to the character
of the irreducible representation $V^{\mn}$. This concludes the proof.
\end{proof}

\subsection{Pair of random Young diagrams $(\lambda,\mu)$}

One of the claims in \cref{theo:main} is that
the random Young diagrams $\lambda,\mu$ are related to each other by
the equality $\mu=\lambda'$.
This result would follow from the decomposition \eqref{eq:decomposition}.
Below we present a short proof based on the original ideas of Howe
\cite[Section 4.1.2]{Howe}.

\medskip

For any permutations $g_1,g_2\in\Sym{p}$ the action of 
$\left(g_2^{-1},g_2^{-1}\right)\in\Sym{p}\times \Sym{p}$
on the external power $\Lambda_p$ coincides with the multiple of identity $(-1)^{g_2}$.
It follows that
\[ \chi_{V_p}(g_1,g_2)= (-1)^{g_2}\ \chi_{V_p} \left(g_1 g_2^{-1}, \operatorname{id} \right).\]
By linearity, it follows that for the minimal central projections
$p_\lambda,p_\mu\in Z\C[\Sym{p}]$
which correspond to the Young diagrams $\lambda,\mu$
we have
\[ \chi_{V_p}(p_\lambda\otimes p_\mu) =\sum_{g_1,g_2\in\Sym{p}} \chi_{V_p}(g_1,g_2) \ p_1(g_1)\ p_2(g_2) =
\chi_{V_p}(p_{\lambda} p_{\mu'} \otimes 1).
\]
The left-hand side is equal to the probability of sampling the pair $(\lambda,\mu)$;
the right hand side vanishes unless $\lambda=\mu'$ which concludes the proof.

\subsection{Stanley character formula}
\begin{remark}
It is easy to use the ideas presented in the above proof
to find a new elementary proof of Stanley's formula 
\cite[Theorem 1]{Stanley2001} for
the character of the symmetric group $\Sym{mn}$ 
which corresponds to the rectangular diagram $n^m$;
for other proofs see also \cite{Feray2010,Feray2011}.
More specifically, one should calculate
the character $\chi_{V_p}$ directly by calculating
the trace \eqref{eq:character1} in the standard basis
of the tensor power $(\C^m\otimes \C^n)^{\otimes p}$
and apply \cref{coro:character-restriction}.
\end{remark}

\section{Outlook}

We have to admit that in \cref{sec:problem} we quoted only \emph{a part}
of the original question of Landsberg; in particular we have skipped the 
following more specific passage.

\begin{quotation}
\emph{%
[\dots] 
I am most interested in the case where $p$ is near $\frac{mn}{2}$. Is there a slowly growing
function $f(n)$ such that partitions with fewer than $f(n)$ steps contribute negligibly?
If so, can the fastest growing such $f$ be determined? 
}

\hfill Joseph M.~Landsberg \cite{Landsberg2012}
\end{quotation}

It is not clear if the ideas presented in this note are sufficient to tackle this
more specific problem. 
For more on this topic see the work of Sevak Mkrtchyan \cite{Mkrtchy2017}.

\section*{Acknowledgments}
Research of PŚ was supported by \emph{Narodowe Centrum Nauki}, grant number 2014/15/B/ST1/00064. Research of GP is partially supported by the NSF. 
We thank Paul Wedrich for pointing out the reference \cite{Howe}.
We thank Vadim Gorin for an interesting discussion.
\cref{fig:romik} has been provided by Dan Romik.

\bibliography{biblio}
\bibliographystyle{alpha}

\end{document}